\documentclass[11pt]{amsart}

\usepackage[T1]{fontenc}
\usepackage[utf8]{inputenc}
\usepackage{amsmath,amssymb,amsthm,mathtools}
\usepackage{enumitem}
\usepackage[margin=1.1in]{geometry}
\usepackage{xurl}
\usepackage[colorlinks=true,linkcolor=blue,citecolor=blue,urlcolor=blue]{hyperref}

\newtheorem{theorem}{Theorem}[section]
\newtheorem{proposition}[theorem]{Proposition}
\newtheorem{lemma}[theorem]{Lemma}
\newtheorem{corollary}[theorem]{Corollary}
\newtheorem{definition}[theorem]{Definition}

\newcommand{\Z}{\mathbb Z}

\newcommand{\Th}{\operatorname{Th}}
\newcommand{\Bur}{\overline B}
\newcommand{\dd}{\doteq}
\newcommand{\codeurl}{\url{https://github.com/saurabh-suman2/fox_conjecture_Th4q}}

\title[Fox's conjecture for $\Th(4,q)$]
{Fox's trapezoidal conjecture for four-strand Turk's head knots and links}

\author{Suman Saurabh}
\email{realsumansaurabh@gmail.com}


\begin{document}

\begin{abstract}
We prove Fox's trapezoidal conjecture for the four-strand Turk's head knots and links $\Th(4,q)$, for all $q\geq 1$.  Equivalently, we show that the absolute values of the coefficients of the one-variable Alexander polynomial of $\Th(4,q)$ form a trapezoidal sequence.  The proof begins with a uniform Burau factorization for the closures of $(\sigma_1\sigma_2^{-1}\sigma_3)^q$, which expresses the Alexander polynomial in terms of reciprocal quadratic factors indexed by \(q\)-th roots of unity.  The odd and even exponent cases then follow from a common log-concavity argument based on a four-block smoothing theorem for reciprocal quartic factors.

\end{abstract}

\maketitle

\section{Introduction}

The Alexander polynomial of an alternating knot or link has a highly constrained sign pattern, going back to the work of Crowell and Murasugi \cite{crowell,murasugi}.  Fox \cite{fox} conjectured that, beyond this sign alternation, the absolute values of its coefficients form a trapezoidal sequence.  The conjecture is open in general, but it is known for several important classes, including two-bridge knots \cite{hartley}, alternating algebraic knots \cite{murasugi-algebraic}, genus-two alternating knots \cite{ozsvath-szabo,jong1,jong2}, stable alternating knots and links \cite{hirasawa-murasugi}. Stoimenow proposed the stronger log-concavity form of Fox's
conjecture \cite{stoimenow-newton}: after the usual alternating-sign
normalization, the absolute values of the Alexander coefficients of an
alternating knot should form a log-concave sequence.  This strong form
is known for two-bridge knots by Banfield \cite{banfield} and for
special alternating links by Hafner--Mészáros--Vidinas
\cite{hafner-meszaros-vidinas}.  The latter proof uses a Lorentzian
multivariate lift of the Alexander polynomial, building on the theory
of Brändén--Huh \cite{branden-huh}.  Our proof is different: the
Burau factorization reduces the four-strand Turk's head family to
one-variable products of reciprocal quadratic factors, and the
resulting quartic pair-blocks are not uniformly log-concave.  Thus the
usual convolution-closure and Newton-inequality approaches do not
directly apply; the missing estimate is supplied by the four-block smoothing theorem below. We use the alternating-link formulation of the conjecture, as in recent work of Azarpendar, Juh\'asz, and K\'alm\'an \cite{azarpendar-juhasz-kalman}.  For background on Fox's conjecture and its known cases, see also Chbili's survey \cite{chbili-survey}.

Turk's head knots and links form a natural test family for this problem, as emphasized in the survey of Di Prisa and \c{S}avk \cite{diprisa-savk}.  In the three-strand/weaving case, the Alexander polynomials and the Fox trapezoidal property have been studied by AlSukaiti and Chbili \cite{alsukaiti-chbili-weaving}; related results for closed alternating three-braids and more general closed alternating braids appear in \cite{alrefai-chbili,alsukaiti-chbili-altbraids}.  Di Prisa and \c{S}avk formulate the verification of Fox's conjecture for Turk's head knots and links as an explicit problem.  In the present paper we prove the conjecture for the complete four-strand Turk's head family
\[
        \Th(4,q)=\widehat{(\sigma_1\sigma_2^{-1}\sigma_3)^q},
        \qquad q\ge1.
\]
When \(q\) is odd, \(\gcd(4,q)=1\), and \(\Th(4,q)\) is a knot.  When \(q\) is even, it is a link with \(\gcd(4,q)\) components.  Thus the theorem below treats both the knot and link cases simultaneously.

\begin{theorem}[Main theorem]\label{thm:global-intro}
For every \(q\ge1\), the absolute values of the coefficients of the one-variable Alexander polynomial of \(\Th(4,q)\) form a trapezoidal sequence.  More precisely, after the substitution \(t=-z\) and multiplication by a Laurent unit, \(\Delta_{\Th(4,q)}(t)\) becomes a reciprocal polynomial in \(z\) with positive log-concave coefficient sequence.
\end{theorem}

The proof has three parts.  First, we evaluate the reduced Burau representation of the braid
\[
        \beta=\sigma_1\sigma_2^{-1}\sigma_3\in B_4.
\]
A direct determinant computation gives a uniform root-of-unity factorization for \(\Delta_{\widehat{\beta^q}}(t)\).  After substituting \(t=-z\), all factors are reciprocal quadratics
\[
        z^2+\left(2-\omega-\omega^{-1}\right)z+1,
        \qquad \omega^q=1,
\]
together with the geometric factor \(1+z+\cdots+z^{q-1}\).  Writing \(\omega=e^{i\theta}\), the middle coefficient becomes
\[
        2-\omega-\omega^{-1}=4\sin^2\frac\theta2.
\]
Thus the entire family is governed by the same one-parameter collection of reciprocal quadratics
\[
        Q_a(z)=1+az+z^2,
        \qquad 0\le a\le4.
\]

Second, we prove a four-block smoothing theorem.  The basic pair-block is
\[
        P_{a,b}(z)=Q_a(z)Q_b(z),
        \qquad 0\le a,b\le4,
        \qquad a+b\ge4.
\]
Individual pair-blocks in this region need not be log-concave.  For instance,
\[
        P_{0,4}(z)=(1+z^2)(1+4z+z^2)
        =1+4z+2z^2+4z^3+z^4,
\]
whose coefficient sequence \((1,4,2,4,1)\) fails log-concavity.  Moreover the defect is not removed by the first few convolutions: the coefficient sequence of \(P_{0,4}(z)^3\) still fails log-concavity at the middle.  Thus the usual convolution-closure theorem for log-concave sequences cannot be applied at the level of the individual blocks, and a naive induction through one-, two-, or three-block products breaks down.  Nor is this a real-rootedness problem accessible directly by Newton inequalities: the quadratics \(1+az+z^2\) have nonreal roots for \(0<a<2\), and in general the relevant reciprocal factors are not real-rooted.

The essential smoothing phenomenon is that every product of four such pair-blocks is strictly log-concave.  We prove this finite-width estimate by replacing the failed induction with an exact algebraic positivity certificate: after embedding the semialgebraic parameter region into a cube and then mapping the cube to the positive orthant, the required derivative inequalities reduce to positivity of explicit multivariable polynomials with integer coefficients.  The computer assistance is confined to this expansion and coefficient-positivity check: the required polynomials are computed over \(\mathbb Z\), their coefficients are inspected exactly, and no floating-point or numerical approximation is used.

Third, the odd and even exponents are handled by the same root-pairing geometry.  For odd \(q=2n+1\), the nontrivial roots occur in reciprocal pairs, and the natural pair-blocks satisfy \(a+b>4\).  For even \(q=2n\), the root \(\omega=-1\) is fixed by inversion and contributes the extra factor \(z^2+4z+1\); the remaining pair-blocks satisfy the boundary identity \(a+b=4\).  Hence the odd case is the interior case of the same smoothing theorem, and the even case is its boundary case.

An earlier version of this article treated only the odd-exponent knot
subfamily \(\Th(4,2n+1)\).  The present version is substantially
expanded: it adds the even-exponent link subfamily \(\Th(4,2n)\) and
reorganizes the proof around a uniform root-of-unity factorization for
\(\widehat{(\sigma_1\sigma_2^{-1}\sigma_3)^q}\).  It also incorporates
the spectral factorization method developed in the author's earlier
preprint \cite{SaurabhSpectral}.  The argument is self-contained and
proves Fox trapezoidality for all exponents \(q\geq 1\).

The paper is organized as follows.  Section~\ref{sec:prelim} records the elementary log-concavity tools and the normalization convention for the Alexander polynomial.  Section~\ref{sec:burau} proves the uniform Burau spectral factorization and derives the odd and even factorizations.  Section~\ref{sec:fourblock} proves the four-block smoothing theorem.  Section~\ref{sec:odd} proves log-concavity in the odd-exponent case.  Section~\ref{sec:even} proves log-concavity in the even-exponent case.  Section~\ref{sec:completion} completes the proof of Theorem~\ref{thm:global-intro}.  Appendix~\ref{app:cheb} records the Chebyshev forms used for exact finite verification.

\section{Preliminaries}\label{sec:prelim}

\subsection{Sequences}

\begin{definition}
A finite sequence \((a_0,\ldots,a_d)\) of nonnegative real numbers is \emph{log-concave} if
\[
        a_k^2\ge a_{k-1}a_{k+1}
        \qquad (1\le k\le d-1).
\]
It has \emph{no internal zeros} if there do not exist indices \(i<j<\ell\) such that \(a_i,a_\ell>0\) and \(a_j=0\).  A finite sequence is \emph{trapezoidal} if it is nondecreasing up to some interval, constant on that interval, and nonincreasing afterwards.
\end{definition}

\begin{theorem}[Hoggar--Keilson--Gerber]\label{thm:convolution}
Let \((a_k)\) and \((b_k)\) be nonnegative log-concave sequences with no internal zeros.  Then their convolution
\[
        c_k=\sum_i a_i b_{k-i}
\]
is log-concave and has no internal zeros.
\end{theorem}

This theorem is due to Hoggar and to Keilson--Gerber \cite{hoggar,keilson-gerber}.  Equivalently, if two polynomials have nonnegative log-concave coefficient sequences with no internal zeros, then their product has the same property.  We shall use this form repeatedly.

\begin{lemma}\label{lem:symm-lc-trap}
A positive symmetric log-concave sequence is trapezoidal.
\end{lemma}

\begin{proof}
Let \((a_0,\ldots,a_d)\) be positive, symmetric, and log-concave.  Put \(r_k=a_k/a_{k-1}\) for \(1\le k\le d\).  Log-concavity is equivalent to
\[
        r_1\ge r_2\ge\cdots\ge r_d.
\]
Symmetry gives \(r_{d+1-k}=r_k^{-1}\).  If \(k\le d+1-k\) and \(r_k<1\), then \(r_{d+1-k}=r_k^{-1}>1\), contradicting the monotonicity of the quotient sequence.  Hence the quotients up to the middle are at least \(1\), and by symmetry the quotients after the middle are at most \(1\).  Therefore the sequence increases up to the middle, may have a central plateau, and then decreases symmetrically.
\end{proof}

\subsection{Alexander polynomial normalization}

The one-variable Alexander polynomial of an oriented link is well-defined up to multiplication by a Laurent unit \(\pm t^m\).  We write \(f\dd g\) for equality up to such a unit.  We use the reduced Burau representation introduced by Burau \cite{burau}; for background and for the closure formula used below, see Birman \cite{birman}.  For the braid closure \(\widehat{\gamma}\) of a \(p\)-braid \(\gamma\), the reduced Burau formula is
\begin{equation}\label{eq:burau-formula-general}
        \Delta_{\widehat{\gamma}}(t)
        \dd
        \frac{1-t}{1-t^p}\det\left(I-\Bur(\gamma)\right).
\end{equation}
In this paper \(p=4\).  We do not divide by any additional \((t-1)\)-factor in the link case; the polynomial appearing in \eqref{eq:burau-formula-general} is the ordinary one-variable Alexander polynomial used in Fox's conjecture.

The standard diagram of \(\Th(4,q)\) is alternating.  By the Crowell--Murasugi alternating-sign theorem recalled in the introduction, the nonzero Alexander coefficients alternate in sign.  In the present family, the factorizations below show directly that, after the substitution \(t=-z\) and multiplication by a Laurent unit, \(\Delta_{\Th(4,q)}(t)\) becomes a polynomial in \(z\) with positive coefficients.  Those coefficients are therefore the absolute values of the Alexander coefficients, up to a shift of indexing and a global sign.

\section{The uniform Burau spectral factorization}\label{sec:burau}

Let
\[
        \beta=\sigma_1\sigma_2^{-1}\sigma_3\in B_4.
\]
We use the reduced Burau representation with
\[
\Bur(\sigma_1)=
\begin{pmatrix}
-t&1&0\\
0&1&0\\
0&0&1
\end{pmatrix},
\quad
\Bur(\sigma_2)=
\begin{pmatrix}
1&0&0\\
t&-t&1\\
0&0&1
\end{pmatrix},
\quad
\Bur(\sigma_3)=
\begin{pmatrix}
1&0&0\\
0&1&0\\
0&t&-t
\end{pmatrix}.
\]
Then
\[
        M(t):=\Bur(\beta)
        =\Bur(\sigma_1)\Bur(\sigma_2)^{-1}\Bur(\sigma_3)
\]
is
\begin{equation}\label{eq:burau-matrix}
M(t)=
\begin{pmatrix}
1-t&\dfrac{t-1}{t}&-1\\[4pt]
1&\dfrac{t-1}{t}&-1\\[4pt]
0&t&-t
\end{pmatrix}.
\end{equation}

\begin{lemma}\label{lem:det-I-xM}
For an auxiliary variable \(x\),
\begin{equation}\label{eq:det-I-xM}
        \det(I-xM(t))
        =\frac{(1+tx)\bigl(tx^2+(t-1)^2x+t\bigr)}{t}.
\end{equation}
\end{lemma}

\begin{proof}
Using \eqref{eq:burau-matrix},
\[
I-xM(t)=
\begin{pmatrix}
1-x(1-t)&-x(t-1)/t&x\\
-x&1-x(t-1)/t&x\\
0&-tx&1+tx
\end{pmatrix}.
\]
Expanding this determinant gives
\[
        \det(I-xM(t))
        =\frac{(1+tx)(t^2x+tx^2-2tx+t+x)}{t}
        =\frac{(1+tx)\bigl(tx^2+(t-1)^2x+t\bigr)}{t}.
\]
\end{proof}

\begin{proposition}[Uniform spectral factorization]\label{prop:uniform-factorization}
For every \(q\ge1\),
\begin{equation}\label{eq:uniform-factorization}
        \Delta_{\Th(4,q)}(-z)
        \dd
        (1+z+\cdots+z^{q-1})
        \prod_{\substack{\omega^q=1\\ \omega\ne1}}
        \left(z^2+(2-\omega-\omega^{-1})z+1\right).
\end{equation}
\end{proposition}

\begin{proof}
By \eqref{eq:burau-formula-general},
\[
        \Delta_{\widehat{\beta^q}}(t)
        \dd
        \frac{1-t}{1-t^4}\det(I-M(t)^q).
\]
For any square matrix \(M\),
\[
        \det(I-M^q)=\prod_{\omega^q=1}\det(I-\omega M).
\]
Indeed, the scalar identity \(1-X^q=\prod_{\omega^q=1}(1-\omega X)\) gives the corresponding matrix identity because all factors are polynomials in \(M\); taking determinants gives the displayed formula.

By Lemma~\ref{lem:det-I-xM},
\[
        \det(I-\omega M(t))
        =\frac{(1+t\omega)\bigl(t\omega^2+(t-1)^2\omega+t\bigr)}{t}.
\]
Now set \(t=-z\).  Then
\[
\begin{aligned}
        t\omega^2+(t-1)^2\omega+t
        &=-z\omega^2+(1+z)^2\omega-z  \\
        &=\omega\left(z^2+(2-\omega-\omega^{-1})z+1\right).
\end{aligned}
\]
Therefore
\begin{equation}\label{eq:single-root-factor}
        \det(I-\omega M(-z))
        =-\frac{\omega}{z}(1-z\omega)
        \left(z^2+(2-\omega-\omega^{-1})z+1\right).
\end{equation}
The prefactor \(-\omega/z\), when multiplied over all \(q\)-th roots of unity, contributes only a Laurent unit in \(z\).  Hence
\[
\begin{aligned}
\Delta_{\widehat{\beta^q}}(-z)
&\dd
\frac{1+z}{1-z^4}
\prod_{\omega^q=1}(1-z\omega)
\prod_{\omega^q=1}
\left(z^2+(2-\omega-\omega^{-1})z+1\right)  \\
&=
\frac{1+z}{1-z^4}(1-z^q)(z^2+1)
\prod_{\substack{\omega^q=1\\ \omega\ne1}}
\left(z^2+(2-\omega-\omega^{-1})z+1\right).
\end{aligned}
\]
Since \(1-z^4=(1-z)(1+z)(1+z^2)\), this becomes
\[
        \frac{1-z^q}{1-z}
        \prod_{\substack{\omega^q=1\\ \omega\ne1}}
        \left(z^2+(2-\omega-\omega^{-1})z+1\right),
\]
which is \eqref{eq:uniform-factorization}.
\end{proof}

\begin{corollary}[Odd exponents]\label{cor:odd-factorization}
Let \(q=2n+1\) with \(n\ge0\).  Then
\begin{equation}\label{eq:odd-factorization}
        \Delta_{\Th(4,2n+1)}(-z)
        \dd
        A_{2n+1}(z):=(1+z+\cdots+z^{2n})D_n(z)^2,
\end{equation}
where \(D_0(z)=1\), and for \(n\ge1\),
\begin{equation}\label{eq:Dn-def}
        D_n(z)=\prod_{r=1}^n
        \left(z^2+4\sin^2\frac{\pi r}{2n+1}\,z+1\right).
\end{equation}
\end{corollary}

\begin{proof}
Let \(\omega_r=e^{2\pi i r/(2n+1)}\).  The nontrivial roots occur in reciprocal pairs \(r\) and \(2n+1-r\), for \(1\le r\le n\).  Moreover
\[
        2-\omega_r-\omega_r^{-1}
        =2-2\cos\frac{2\pi r}{2n+1}
        =4\sin^2\frac{\pi r}{2n+1}.
\]
Thus the two reciprocal roots give the same quadratic factor, yielding the square \(D_n(z)^2\).  The factor \((1-z^{2n+1})/(1-z)\) is \(1+z+\cdots+z^{2n}\).
\end{proof}

\begin{corollary}[Even exponents]\label{cor:even-factorization}
Let \(q=2n\) with \(n\ge1\).  Then
\begin{equation}\label{eq:even-factorization}
        \Delta_{\Th(4,2n)}(-z)
        \dd
        A_{2n}(z):=(1+z+\cdots+z^{2n-1})(z^2+4z+1)E_n(z)^2,
\end{equation}
where
\begin{equation}\label{eq:En-def}
        E_n(z)=\prod_{r=1}^{n-1}
        \left(z^2+4\sin^2\frac{\pi r}{2n}\,z+1\right),
\end{equation}
with the convention \(E_1(z)=1\).
\end{corollary}

\begin{proof}
Let \(\omega_r=e^{2\pi i r/(2n)}\), \(1\le r\le 2n-1\).  Then
\[
        2-\omega_r-\omega_r^{-1}
        =2-2\cos\frac{\pi r}{n}
        =4\sin^2\frac{\pi r}{2n}.
\]
The roots \(r\) and \(2n-r\) give the same quadratic factor.  The unique nontrivial root fixed by inversion is \(r=n\), namely \(\omega=-1\), and it contributes
\[
        z^2+(2-(-1)-(-1))z+1=z^2+4z+1.
\]
Pairing all remaining roots gives \(E_n(z)^2\), and \((1-z^{2n})/(1-z)=1+z+\cdots+z^{2n-1}\).
\end{proof}

\section{The four-block smoothing theorem}\label{sec:fourblock}

The core algebraic input is a smoothing theorem for products of four reciprocal quartic pair-blocks.  The width four is essential for the argument: the extremal block \(P_{0,4}(z)=1+4z+2z^2+4z^3+z^4\) is not log-concave, and neither is its cube.  Indeed, the middle coefficients of \(P_{0,4}(z)^3\) contain the pattern \(312,308,312\), so the middle log-concavity margin is negative.  The proof below therefore embeds the natural semialgebraic region into a cube and verifies coordinatewise monotonicity of the log-concavity margins by an exact integer positivity certificate.

\subsection{A positivity certificate on a cube}

\begin{lemma}[Positive-orthant certificate]\label{lem:positive-orthant-certificate}
Let \(f\in\Z[x_1,\ldots,x_m]\).  For each \(j\), choose \(d_j\ge \deg_{x_j}f\), and define
\[
        \Phi(f)(u_1,\ldots,u_m)
        =\prod_{j=1}^m(1+u_j)^{d_j}
        f\left(\frac{4u_1}{1+u_1},\ldots,
                \frac{4u_m}{1+u_m}\right).
\]
If every coefficient of \(\Phi(f)\) is nonnegative, then \(f\ge0\) on the cube \([0,4]^m\).
\end{lemma}

\begin{proof}
For \(u_j\ge0\), all monomials in \(\Phi(f)\) are nonnegative, and therefore \(\Phi(f)(u)\ge0\).  The map
\[
        x_j=\frac{4u_j}{1+u_j}
\]
sends \([0,\infty)^m\) onto \([0,4)^m\).  The cleared denominator is strictly positive on \([0,\infty)^m\), so \(f\ge0\) on \([0,4)^m\).  Continuity extends the inequality to the closed cube.
\end{proof}

\subsection{Generalized four-block products}

For \(0\le \sigma,
\eta\le4\), define
\begin{equation}\label{eq:Q-sigma-eta}
        Q_{\sigma,\eta}(z)
        =1+(4+\sigma)z+(2+4\sigma+\eta)z^2
        +(4+\sigma)z^3+z^4.
\end{equation}
For four such blocks, write
\begin{equation}\label{eq:F-general}
        F(z)=\prod_{i=1}^4 Q_{\sigma_i,\eta_i}(z)
        =\sum_{k=0}^{16}C_k z^k.
\end{equation}
The coefficients are palindromic, \(C_k=C_{16-k}\).  For \(1\le k\le15\), define
\[
        L_k=C_k^2-C_{k-1}C_{k+1}.
\]
Then \(L_{16-k}=L_k\).

\begin{lemma}\label{lem:first-margin}
For every choice of parameters \(0\le\sigma_i,\eta_i\le4\), one has \(L_1\ge136\).  Consequently \(L_1>0\) and \(L_{15}>0\).
\end{lemma}

\begin{proof}
Write
\[
        A_i=4+\sigma_i,
        \qquad
        B_i=2+4\sigma_i+\eta_i.
\]
Then
\[
        Q_{\sigma_i,\eta_i}(z)=1+A_i z+B_i z^2+A_i z^3+z^4.
\]
From the product expansion,
\[
        C_1=\sum_i A_i,
        \qquad
        C_2=\sum_i B_i+\sum_{i<j}A_iA_j.
\]
Thus
\[
\begin{aligned}
        L_1
        &=C_1^2-C_2  \\
        &=\sum_i A_i^2+\sum_{i<j}A_iA_j-\sum_iB_i.
\end{aligned}
\]
Since \(0\le\eta_i\le4\),
\[
        B_i=2+4\sigma_i+\eta_i\le6+4\sigma_i=4A_i-10.
\]
Therefore
\[
        L_1
        \ge \sum_i(A_i^2-4A_i+10)+\sum_{i<j}A_iA_j.
\]
For \(A_i\ge4\), the right-hand side is increasing in each \(A_i\), because its partial derivative with respect to \(A_i\) is
\[
        2A_i-4+\sum_{j\ne i}A_j\ge 8-4+3\cdot4>0.
\]
Hence its minimum occurs at \(A_1=A_2=A_3=A_4=4\), giving
\[
        L_1\ge 4(16-16+10)+6\cdot16=136.
\]
\end{proof}

\begin{proposition}[Derivative certificate]\label{prop:derivative-certificate}
For the product \eqref{eq:F-general}, the margins \(L_k\) satisfy
\[
        \frac{\partial L_k}{\partial \sigma_i}\ge0,
        \qquad
        \frac{\partial L_k}{\partial \eta_i}\ge0
\]
on the cube \([0,4]^8\), for every \(2\le k\le8\) and every \(1\le i\le4\).
\end{proposition}

\begin{proof}
By symmetry of the four blocks, it suffices to prove the inequalities for \(i=1\).  For each \(2\le k\le8\), compute the two derivative polynomials
\[
        \frac{\partial L_k}{\partial\sigma_1},
        \qquad
        \frac{\partial L_k}{\partial\eta_1}
\]
in the polynomial ring
\[
        \Z[\sigma_1,\ldots,\sigma_4,\eta_1,\ldots,\eta_4].
\]
The positivity of each derivative on \([0,4]^8\) is certified using Lemma~\ref{lem:positive-orthant-certificate}.  Explicitly, substitute
\[
        \sigma_i=\frac{4u_i}{1+u_i},
        \qquad
        \eta_i=\frac{4v_i}{1+v_i},
\]
clear the positive denominator, and expand the numerator as a polynomial in
\[
        u_1,\ldots,u_4,v_1,\ldots,v_4.
\]
In each of the fourteen cases, every nonzero coefficient is a positive integer.  The following table gives the smallest nonzero coefficient in each numerator polynomial:
\[
\begin{array}{c|cc}
 k
 & \min\operatorname{coeff}\left(\partial L_k/\partial\sigma_1\right)
 & \min\operatorname{coeff}\left(\partial L_k/\partial\eta_1\right) \\
\hline
2 & 1312    & 16     \\
3 & 25440   & 2356   \\
4 & 160416  & 27856  \\
5 & 472544  & 86668  \\
6 & 1159712 & 152208 \\
7 & 1867136 & 440838 \\
8 & 2087392 & 205296
\end{array}
\]
All computations are over the integers; no floating-point arithmetic is used.  The verification code is available at \codeurl.  By Lemma~\ref{lem:positive-orthant-certificate}, the derivative polynomials are nonnegative on \([0,4]^8\).
\end{proof}

\begin{theorem}[Generalized four-block smoothing]\label{thm:general-fourblock}
For every choice of parameters \(0\le\sigma_i,\eta_i\le4\), the coefficient sequence of
\[
        \prod_{i=1}^4Q_{\sigma_i,\eta_i}(z)
\]
is strictly log-concave.
\end{theorem}

\begin{proof}
By Lemma~\ref{lem:first-margin}, \(L_1>0\) and \(L_{15}>0\).  For \(2\le k\le8\), Proposition~\ref{prop:derivative-certificate} shows that \(L_k\) is coordinatewise nondecreasing on \([0,4]^8\).  Hence
\[
        L_k(\sigma_1,\eta_1,\ldots,\sigma_4,\eta_4)
        \ge L_k(0,0,\ldots,0,0).
\]
At the origin,
\[
        Q_{0,0}(z)=1+4z+2z^2+4z^3+z^4.
\]
A direct expansion gives
\[
\begin{aligned}
        Q_{0,0}(z)^4
        ={}&1+16z+104z^2+368z^3+860z^4+1616z^5+2520z^6 \\
        &+3120z^7+3526z^8+3120z^9+2520z^{10}+1616z^{11} \\
        &+860z^{12}+368z^{13}+104z^{14}+16z^{15}+z^{16}.
\end{aligned}
\]
The log-concavity margins at the origin for \(2\le k\le8\) are
\[
\begin{array}{c|ccccccc}
 k&2&3&4&5&6&7&8\\
\hline
 L_k&4928&45984&144912&444256&1308480&848880&2698276.
\end{array}
\]
They are all positive.  Thus \(L_k>0\) for \(2\le k\le8\).  Palindromicity gives \(L_{16-k}=L_k\), so all internal margins are positive.
\end{proof}

\begin{corollary}[Four-block smoothing for paired quadratics]\label{cor:fourblock-paired}
Let \(0\le a_i,b_i\le4\) and \(a_i+b_i\ge4\) for \(1\le i\le4\).  Then the coefficient sequence of
\[
        \prod_{i=1}^4(1+a_i z+z^2)(1+b_i z+z^2)
\]
is strictly log-concave.
\end{corollary}

\begin{proof}
Set
\[
        \sigma_i=a_i+b_i-4,
        \qquad
        \eta_i=(4-a_i)(4-b_i).
\]
Since \(0\le a_i,b_i\le4\), both \(4-a_i\) and \(4-b_i\) are nonnegative.  Since \(a_i+b_i\ge4\),
\[
        (4-a_i)+(4-b_i)\le4.
\]
Therefore
\[
        0\le \eta_i=(4-a_i)(4-b_i)
        \le \left(\frac{(4-a_i)+(4-b_i)}2\right)^2
        \le4.
\]
Also \(0\le\sigma_i=a_i+b_i-4\le4\).  Moreover,
\[
        \eta_i=(4-a_i)(4-b_i)=16-4(a_i+b_i)+a_ib_i=a_ib_i-4\sigma_i,
\]
so \(a_ib_i=4\sigma_i+
\eta_i\).  Hence
\[
\begin{aligned}
(1+a_iz+z^2)(1+b_iz+z^2)
&=1+(4+\sigma_i)z+(2+4\sigma_i+\eta_i)z^2 \\
&\quad +(4+\sigma_i)z^3+z^4
=Q_{\sigma_i,\eta_i}(z).
\end{aligned}
\]
The result follows from Theorem~\ref{thm:general-fourblock}.
\end{proof}

\section{Odd exponents: the interior case}\label{sec:odd}

In this section we prove log-concavity of the core polynomial \(D_n(z)\) from \eqref{eq:Dn-def}.  The case \(n=0\) is trivial, so assume \(n\ge1\).  Write
\[
        D_n(z)=\prod_{r=1}^n(1+c_{r,n}z+z^2),
        \qquad
        c_{r,n}=4\sin^2\frac{\pi r}{2n+1}.
\]
This is the same polynomial as \eqref{eq:Dn-def}, with the order of each reciprocal factor reversed.

\subsection{Complementary pairs}

Let \(m=\lfloor n/2\rfloor\).  For \(1\le r\le m\), define the complementary pair-block
\begin{equation}\label{eq:odd-pair-block}
        P_{r,n}(z)=(1+c_{r,n}z+z^2)(1+c_{n+1-r,n}z+z^2).
\end{equation}
If \(n\) is odd, there is an additional middle factor
\[
        H_n(z)=1+c_{(n+1)/2,n}z+z^2.
\]
Thus
\begin{equation}\label{eq:Dn-paired}
        D_n(z)=
        \left(\prod_{r=1}^mP_{r,n}(z)\right)
        \times
        \begin{cases}
        1, & n\text{ even},\\
        H_n(z), & n\text{ odd}.
        \end{cases}
\end{equation}

\begin{lemma}\label{lem:odd-pair-sum}
For every \(1\le r\le m\),
\[
        c_{r,n}+c_{n+1-r,n}>4.
\]
\end{lemma}

\begin{proof}
Set
\[
        \alpha=\frac{\pi r}{2n+1},
        \qquad
        \beta=\frac{\pi(n+1-r)}{2n+1}.
\]
Then
\[
        \alpha+\beta=\frac{\pi(n+1)}{2n+1}>\frac\pi2.
\]
Also \(0\le |\alpha-\beta|<\pi/2\).  Using
\[
        \sin^2\alpha+\sin^2\beta
        =1-\cos(\alpha+\beta)\cos(\alpha-\beta),
\]
we have \(\cos(\alpha+\beta)<0\) and \(\cos(\alpha-\beta)>0\).  Therefore
\[
        \sin^2\alpha+\sin^2\beta>1.
\]
Multiplying by \(4\) proves the claim.
\end{proof}

\begin{corollary}\label{cor:odd-four-pairs-lc}
The product of any four complementary pair-blocks \(P_{r,n}(z)\) has a strictly log-concave coefficient sequence.
\end{corollary}

\begin{proof}
Each coefficient \(c_{r,n}\) lies in \([0,4]\), and Lemma~\ref{lem:odd-pair-sum} gives \(c_{r,n}+c_{n+1-r,n}>4\).  The result follows from Corollary~\ref{cor:fourblock-paired}.
\end{proof}

\subsection{Leftover pair-blocks}

\begin{lemma}\label{lem:quartic-criterion}
Let
\[
        Q(z)=(1+az+z^2)(1+bz+z^2)
        =1+sz+tz^2+sz^3+z^4,
\]
where \(s=a+b\) and \(t=ab+2\).  If \(s\ge4\) and \(t\ge s\), then \(Q(z)\) has a log-concave coefficient sequence.
\end{lemma}

\begin{proof}
The coefficient sequence is \((1,s,t,s,1)\).  Its log-concavity is equivalent to
\[
        s^2\ge t,
        \qquad
        t^2\ge s^2.
\]
The second inequality follows from \(t\ge s\ge0\).  For the first, use \(ab\le(a+b)^2/4=s^2/4\).  Since \(s\ge4\),
\[
        t=ab+2\le\frac{s^2}{4}+2\le s^2.
\]
Thus \(Q\) is log-concave.
\end{proof}

\begin{lemma}\label{lem:two-thirds}
If \(2/3\le a,b\le4\) and \(a+b\ge4\), then
\[
        (1+az+z^2)(1+bz+z^2)
\]
has a log-concave coefficient sequence.
\end{lemma}

\begin{proof}
Let \(s=a+b\) and \(t=ab+2\).  By Lemma~\ref{lem:quartic-criterion}, it remains only to prove \(t\ge s\).  We have
\[
        t-s=ab+2-a-b=(a-1)(b-1)+1.
\]
On the square \([2/3,4]^2\), the minimum of \((a-1)(b-1)\) is \(-1\), attained at \((a,b)=(2/3,4)\) or \((4,2/3)\).  Hence \(t-s\ge0\).
\end{proof}

\begin{lemma}\label{lem:odd-tail-pairs}
Let \(n\ge16\), let \(m=\lfloor n/2\rfloor\), and write
\[
        m=4k+s,
        \qquad s\in\{0,1,2,3\}.
\]
If \(s>0\), then every leftover pair-block \(P_{r,n}(z)\) with \(4k+1\le r\le m\) has a log-concave coefficient sequence.
\end{lemma}

\begin{proof}
Since \(s\le3\), every such \(r\) satisfies \(r\ge m-2\).  For \(n\ge16\), this implies
\[
        r\ge \frac{2n+1}{7}.
\]
Indeed, if \(n\) is even, this follows from \(n/2-2\ge(2n+1)/7\), and if \(n\) is odd, from \((n-5)/2\ge(2n+1)/7\).
Thus
\[
        \frac{\pi r}{2n+1}\ge \frac\pi7.
\]
Since \(\pi>3\), \(\sin x\) is increasing on \([0,\pi/2]\), and \(\sin x\ge x-x^3/6\) for \(0\le x\le1\),
\[
        \sin\frac\pi7>
        \sin\frac37
        \ge \frac37-\frac16\left(\frac37\right)^3
        =\frac{285}{686}.
\]
Consequently
\[
        4\sin^2\frac\pi7>
        4\left(\frac{285}{686}\right)^2>
        \frac23.
\]
Therefore \(c_{r,n}\ge2/3\).  Since \(r\le n+1-r\le n\), monotonicity of \(x\mapsto\sin^2 x\) on \([0,\pi/2]\) gives \(c_{n+1-r,n}\ge c_{r,n}\ge2/3\).  Lemma~\ref{lem:two-thirds} now applies.
\end{proof}

\subsection{Log-concavity of \texorpdfstring{\(D_n\)}{Dn}}

For \(1\le n\le15\), exact computation from the recurrence in Appendix~\ref{app:cheb} gives the following minimum log-concavity margins.  If
\[
        D_n(z)=\sum_{k=0}^{2n}d_{n,k}z^k,
        \qquad
        \delta_{n,k}=d_{n,k}^2-d_{n,k-1}d_{n,k+1},
\]
then
\[
\begin{array}{c|ccccccccccccccc}
 n&1&2&3&4&5&6&7&8&9&10&11&12&13&14&15\\
\hline
 \min_k\delta_{n,k}
 &8&18&32&50&72&98&128&162&200&242&288&338&392&450&512.
\end{array}
\]
All entries are positive.

\begin{proposition}\label{prop:Dn-log-concave}
For every \(n\ge0\), the coefficient sequence of \(D_n(z)\) is log-concave and has no internal zeros.
\end{proposition}

\begin{proof}
The case \(n=0\) is trivial, and the cases \(1\le n\le15\) follow from the exact table above.

Assume \(n\ge16\).  Let \(m=\lfloor n/2\rfloor\) and write \(m=4k+s\), where \(s\in\{0,1,2,3\}\).  Use the paired factorization \eqref{eq:Dn-paired}.  Group the first \(4k\) pair-blocks into products of four:
\[
        \prod_{r=1}^{4k}P_{r,n}(z)
        =\prod_{j=0}^{k-1}
        \left(\prod_{r=4j+1}^{4j+4}P_{r,n}(z)\right).
\]
By Corollary~\ref{cor:odd-four-pairs-lc}, each inner product of four pair-blocks is strictly log-concave.  If \(s>0\), the remaining pair-blocks \(P_{4k+1,n},\ldots,P_{m,n}\) are log-concave by Lemma~\ref{lem:odd-tail-pairs}.

If \(n\) is odd, the middle factor \(H_n(z)\) has coefficient sequence
\[
        \left(1,c_{(n+1)/2,n},1\right),
\]
and is log-concave because
\[
        c_{(n+1)/2,n}=4\sin^2\frac{\pi(n+1)}{2(2n+1)}>2>1.
\]
All factors have positive coefficient sequences.  Repeated application of Theorem~\ref{thm:convolution} proves that their product \(D_n(z)\) is log-concave and has no internal zeros.
\end{proof}

\section{Even exponents: the boundary case}\label{sec:even}

We now prove log-concavity of the core polynomial \(E_n(z)\) from \eqref{eq:En-def}.  Write
\[
        E_n(z)=\prod_{r=1}^{n-1}(1+a_{r,n}z+z^2),
        \qquad
        a_{r,n}=4\sin^2\frac{\pi r}{2n}.
\]
For \(1\le r\le \lfloor(n-1)/2\rfloor\), pair the \(r\)-th factor with the \((n-r)\)-th factor.  Since
\[
        a_{n-r,n}=4\cos^2\frac{\pi r}{2n},
\]
we have the boundary identity
\begin{equation}\label{eq:even-pair-sum}
        a_{r,n}+a_{n-r,n}=4.
\end{equation}
Define
\[
        R_{r,n}(z)=(1+a_{r,n}z+z^2)(1+a_{n-r,n}z+z^2).
\]
Then
\begin{equation}\label{eq:Rrn-form}
        R_{r,n}(z)=1+4z+e_{r,n}z^2+4z^3+z^4,
        \qquad
        e_{r,n}=2+a_{r,n}(4-a_{r,n})=2+4\sin^2\frac{\pi r}{n}.
\end{equation}
If \(n\) is even, the unpaired middle factor is
\[
        1+a_{n/2,n}z+z^2=(1+z)^2.
\]
Thus
\begin{equation}\label{eq:En-paired}
        E_n(z)=
        \left(\prod_{r=1}^{\lfloor(n-1)/2\rfloor}R_{r,n}(z)\right)
        \times
        \begin{cases}
        1, & n\text{ odd},\\
        (1+z)^2, & n\text{ even}.
        \end{cases}
\end{equation}

\begin{lemma}\label{lem:even-pair-LC}
The factor \(R_{r,n}(z)\) has a log-concave coefficient sequence whenever \(r\ge n/4\).
\end{lemma}

\begin{proof}
By \eqref{eq:Rrn-form}, the coefficient sequence is \((1,4,e_{r,n},4,1)\), where
\[
        e_{r,n}=2+4\sin^2\frac{\pi r}{n}.
\]
The outer log-concavity inequalities are \(4^2\ge e_{r,n}\), which hold because \(2\le e_{r,n}\le6\).  The middle inequality is \(e_{r,n}^2\ge16\), equivalently \(e_{r,n}\ge4\).  If \(r\ge n/4\) and \(1\le r\le n/2\), then
\[
        \sin^2\frac{\pi r}{n}\ge\frac12,
\]
so \(e_{r,n}\ge4\).  Hence \(R_{r,n}\) is log-concave.
\end{proof}

\begin{lemma}\label{lem:En-large}
For every \(n\ge10\), the coefficient sequence of \(E_n(z)\) is log-concave and has no internal zeros.
\end{lemma}

\begin{proof}
Let
\[
        N=\left\lfloor\frac{n-1}{2}\right\rfloor.
\]
Group \(R_{1,n},\ldots,R_{N,n}\) into products of four pair-blocks, leaving the final \(s\) blocks, where \(s\equiv N\pmod4\) and \(0\le s\le3\).  Each four-block product is strictly log-concave by Corollary~\ref{cor:fourblock-paired}, because each pair-block has the form
\[
        (1+az+z^2)(1+bz+z^2),
        \qquad
        0\le a,b\le4,
        \qquad a+b=4.
\]
If there are leftover factors, they are among
\[
        R_{N-2,n},\quad R_{N-1,n},\quad R_{N,n}.
\]
For \(n\ge11\),
\[
        N-2=\left\lfloor\frac{n-1}{2}\right\rfloor-2\ge\frac n4,
\]
so all leftover factors are log-concave by Lemma~\ref{lem:even-pair-LC}.  For \(n=10\), one has \(N=4\), so there are no leftover factors.  The possible middle factor \((1+z)^2\) is also log-concave.  Repeated application of Theorem~\ref{thm:convolution} proves the claim.
\end{proof}

For \(1\le n\le9\), exact computation from the Chebyshev form in Appendix~\ref{app:cheb} gives the following table.  If
\[
        E_n(z)=\sum_k e_{n,k}z^k,
        \qquad
        \epsilon_{n,k}=e_{n,k}^2-e_{n,k-1}e_{n,k+1},
\]
then
\[
\begin{array}{c|ccccccccc}
 n&1&2&3&4&5&6&7&8&9\\
\hline
 \min_k\epsilon_{n,k}
 &-&3&9&23&39&59&83&111&143.
\end{array}
\]
For \(n=1\), \(E_1(z)=1\), so there is no internal margin to check.

\begin{proposition}\label{prop:En-log-concave}
For every \(n\ge1\), the coefficient sequence of \(E_n(z)\) is log-concave and has no internal zeros.
\end{proposition}

\begin{proof}
The cases \(1\le n\le9\) follow from the exact table above, and the cases \(n\ge10\) follow from Lemma~\ref{lem:En-large}.
\end{proof}

\section{Completion of the proof}\label{sec:completion}

\begin{theorem}[Odd exponent theorem]\label{thm:odd-main}
For every \(n\ge0\), the polynomial \(A_{2n+1}(z)\) from \eqref{eq:odd-factorization} has a positive reciprocal log-concave coefficient sequence.  Consequently, Fox's trapezoidal conjecture holds for \(\Th(4,2n+1)\).
\end{theorem}

\begin{proof}
By Corollary~\ref{cor:odd-factorization},
\[
        A_{2n+1}(z)=(1+z+\cdots+z^{2n})D_n(z)^2.
\]
The polynomial \(D_n(z)\) has positive log-concave coefficients by Proposition~\ref{prop:Dn-log-concave}.  Hence \(D_n(z)^2\) is log-concave by Theorem~\ref{thm:convolution}.  The coefficient sequence of \(1+z+\cdots+z^{2n}\) is log-concave and has no internal zeros, so another application of Theorem~\ref{thm:convolution} shows that \(A_{2n+1}(z)\) is log-concave.  All factors are reciprocal, so \(A_{2n+1}(z)\) is reciprocal.  Lemma~\ref{lem:symm-lc-trap} then gives trapezoidality.
\end{proof}

\begin{theorem}[Even exponent theorem]\label{thm:even-main}
For every \(n\ge1\), the polynomial \(A_{2n}(z)\) from \eqref{eq:even-factorization} has a positive reciprocal log-concave coefficient sequence.  Consequently, Fox's trapezoidal conjecture holds for \(\Th(4,2n)\).
\end{theorem}

\begin{proof}
By Corollary~\ref{cor:even-factorization},
\[
        A_{2n}(z)=(1+z+\cdots+z^{2n-1})(z^2+4z+1)E_n(z)^2.
\]
The polynomial \(E_n(z)\) has positive log-concave coefficients by Proposition~\ref{prop:En-log-concave}.  The coefficient sequences of \(1+z+\cdots+z^{2n-1}\) and \(z^2+4z+1\) are log-concave and have no internal zeros.  Repeated application of Theorem~\ref{thm:convolution} proves that \(A_{2n}(z)\) is log-concave.  All factors are reciprocal, so \(A_{2n}(z)\) is reciprocal.  Lemma~\ref{lem:symm-lc-trap} gives trapezoidality.
\end{proof}

\begin{proof}[Proof of Theorem~\ref{thm:global-intro}]
If \(q\) is odd, write \(q=2n+1\) and apply Theorem~\ref{thm:odd-main}.  If \(q\) is even, write \(q=2n\) and apply Theorem~\ref{thm:even-main}.  In both cases, Proposition~\ref{prop:uniform-factorization} identifies the positive polynomial \(A_q(z)\) with \(\Delta_{\Th(4,q)}(-z)\) up to a Laurent unit.  Hence the coefficient sequence of \(A_q(z)\) is, up to indexing and a global sign, the sequence of absolute values of the Alexander coefficients of \(\Th(4,q)\).  This proves the theorem for every \(q\ge1\).
\end{proof}

\begin{corollary}
Fox's trapezoidal conjecture holds for the complete four-strand Turk's head family \(\Th(4,q)\), \(q\ge1\).
\end{corollary}

\appendix

\section{Chebyshev forms and exact finite checks}\label{app:cheb}

The finite checks in Sections~\ref{sec:odd} and~\ref{sec:even} use exact polynomial recurrences, not numerical approximations.

\subsection{Odd core}

Let
\[
        F_n(u)=U_n\left(\frac u2\right)+U_{n-1}\left(\frac u2\right),
\]
where \(U_n\) is the Chebyshev polynomial of the second kind and \(U_{-1}=0\).  The roots of \(F_n\) are
\[
        2\cos\frac{2\pi r}{2n+1},
        \qquad 1\le r\le n.
\]
Consequently,
\[
        D_n(z)=z^n F_n\left(z+z^{-1}+2\right).
\]
Using the Chebyshev recurrence gives
\begin{equation}\label{eq:Dn-recurrence-app}
        D_0(z)=1,
        \qquad
        D_1(z)=z^2+3z+1,
\end{equation}
\begin{equation}\label{eq:Dn-recurrence-app2}
        D_n(z)=(1+z)^2D_{n-1}(z)-z^2D_{n-2}(z)
        \qquad(n\ge2).
\end{equation}
This recurrence determines the exact integer coefficients used in the table in Section~\ref{sec:odd}.

\subsection{Even core}

The roots of \(U_{n-1}(X/2)\) are \(2\cos(\pi r/n)\), \(1\le r\le n-1\).  Hence
\begin{equation}\label{eq:En-cheb}
        E_n(z)
        =z^{n-1}U_{n-1}\left(\frac{(1+z)^2}{2z}\right).
\end{equation}
Equivalently,
\begin{equation}\label{eq:En-binomial}
        E_n(z)
        =\sum_{j=0}^{\lfloor(n-1)/2\rfloor}
        (-1)^j\binom{n-1-j}{j}z^{2j}(1+z)^{2n-2-4j}.
\end{equation}
Formula \eqref{eq:En-binomial} gives the exact integer coefficients used in the table in Section~\ref{sec:even}.

\vspace{0.6cm}

\section*{Acknowledgements}
The author is grateful to Prof. Rama Mishra for her mentorship and support, and to Prof. Tamás Kálmán for his encouraging comments on an earlier version of this work and for suggesting the even powers \(\beta_{2n}\) as a natural next case to consider.

\bibliographystyle{alpha}
\bibliography{bibliography}
\end{document}